\documentclass[11pt]{amsart}
\usepackage[latin1]{inputenc}
\usepackage{amsmath,amsthm, amscd, amssymb, amsfonts}
\usepackage[dvips]{graphicx}
\numberwithin{equation}{section}

\theoremstyle{plain}

\newtheorem{theorem}{Theorem}[section]

\newtheorem{proposition}[theorem]{Proposition}
\newtheorem{lemma}[theorem]{Lemma}

\theoremstyle{definition}

\newtheorem{ejem}[theorem]{Example}

\theoremstyle{remark}
\newtheorem{remark}[theorem]{Remark}

\newcommand{\1}{\textbf{1}}

\newcommand{\A}{{\mathcal A}}
\newcommand{\B}{{\mathcal B}}
\newcommand{\C}{{\mathcal C}}
\newcommand{\D}{{\mathcal D}}

\newcommand{\TY}{{\mathcal {TY}}}
\newcommand\id{\operatorname{id}}

\newcommand\ad{\operatorname{ad}}

\newcommand\Rep{\operatorname{Rep}}

\newcommand\FPdim{\operatorname{FPdim}}

\newcommand\vect{\operatorname{Vec}}

\newcommand\Irr{\operatorname{Irr}}
\newcommand\cd{\operatorname{c.d.}}

\begin{document}
\title[Solvability of a class of  braided fusion categories]{Solvability of a class of braided fusion categories}
\author{Sonia Natale}
\author{Julia Yael Plavnik}
\address{Facultad de Matem\'atica, Astronom\'\i a y F\'\i sica,
Universidad Nacional de C\'ordoba, CIEM -- CONICET, (5000) Ciudad
Universitaria, C\'ordoba, Argentina}
\email{natale@famaf.unc.edu.ar, plavnik@famaf.unc.edu.ar
\newline \indent \emph{URL:}\/ http://www.famaf.unc.edu.ar/$\sim$natale}
\thanks{The research of S. N. was partially supported by CONICET  and Secyt-UNC. The research of J. P. was partially supported by CONICET, ANPCyT  and Secyt-UNC} 
\date{May 10, 2012}

\begin{abstract} We show that a weakly integral braided fusion category $\C$ such that every simple object of
$\C$ has Frobenius-Perron dimension $\leq 2$ is solvable. In
addition, we prove that such a fusion category  is
group-theoretical in the extreme case where
the universal grading group of $\C$ is trivial.
\end{abstract}

\maketitle

\section{Introduction and main results}
Let $k$ be an algebraically closed field of characteristic zero.
A fusion category over $k$ is a semisimple tensor category over $k$ having finitely many isomorphism
classes of simple objects.
In this paper we consider the problem of giving structural results
of a fusion category $\mathcal C$ under restrictions on the set $\cd
(\C)$ of Frobenius-Perron dimensions of its simple objects.

Results of this type were obtained in the paper \cite{NP}. For
instance, we showed in \cite[Theorem 7.3]{NP} that under the
assumption that $\C$ is braided odd-dimensional and $\cd(\C)
\subseteq \{p^m:\, m \geq 0\}$, where $p$ is a (necessarily odd)
prime number, then $\C$ is solvable. Also, the same is true when
$\C = \Rep H$, where $H$ is a semisimple quasitriangular Hopf
algebra and  $\cd (\mathcal C) = \{1, 2\}$ \cite[Theorem
6.12]{NP}.

Using results of the paper \cite{BN}, we also showed  in
\cite[Theorem 6.4]{NP} that if $\C = \Rep H$, where $H$ is any
semisimple Hopf algebra, and  $\cd(\C) \subseteq \{ 1, 2\}$, then
$\C$ is weakly group-theoretical, and furthermore, it is
group-theoretical if $\C$ coincides with the adjoint subcategory
$\C_{\ad}$.

Our main results are the following theorems. Recall that a fusion
category $\C$ is called \emph{weakly integral} if the
Frobenius-Perron dimension of $\C$ is a natural integer.

\begin{theorem}\label{soluble}
Let $\C$ be a weakly integral braided fusion category such that
$\FPdim X \leq 2$, for all simple object $X$ of $\C$. Then $\C$ is
solvable.
\end{theorem}

Theorem \ref{soluble} extends the   previous result for semisimple
quasitriangular Hopf algebras mentioned above. It implies in
particular that every weakly integral braided fusion category with
Frobenius-Perron dimensions of simple objects at most $2$ is
weakly group-theoretical. This gives some further support to the
conjecture that every weakly integral fusion category is weakly
group-theoretical. See \cite[Question 2]{ENO2}.

It is known that a nilpotent braided  fusion category, which is in addition integral (that is, $\cd(\C) \subseteq \mathbb Z_+$)  is always group-theoretical \cite[Theorem 6.10]{DGNO}. We also show that the same conclusion is true in the opposite extreme case:

\begin{theorem}\label{gp-ttic} Let $\C$ be a weakly integral braided fusion category such that
$\FPdim X \leq 2$, for all simple object $X$ of $\C$.
Suppose that the universal grading group of $\C$ is trivial.
Then $\C$ is group-theoretical.
\end{theorem}

Theorems \ref{soluble} and \ref{gp-ttic} are proved in Section
\ref{pruebas}.  Our proofs rely on the results of Naidu and Rowell
\cite{NaR} for the case where $\mathcal C$ is integral and has a
faithful self-dual simple object of Frobenius-Perron dimension
$2$.

Being group-theoretical, a braided fusion category $\C$ satisfying the assumptions of Theorem \ref{gp-ttic}, has the so called property \textbf{F}, namely,  all asso\-ciated braid group representations on the tensor powers of objects of $\C$ factor over finite groups. See \cite[Corollary 4.4]{ERW}. It is conjectured that every braided weakly integral fusion category does have property \textbf{F} \cite{NaR}. This conjecture has been proved for braided fusion categories $\C$ with $\cd(\C) = \{ 1,
2\}$ such that all objects of $\C$ are self-dual or $\C$ is generated by a self-dual simple object \cite[Corollary 4.3 and Remark 4.4]{NaR}.

The paper is organized as follows. In Section
\ref{preliminaries} we recall the main facts and terminology about
fusion and braided fusion categories used throughout. In Section
\ref{examples} we discuss some families of (integral) examples that appear in
the literature. We also recall in this section the results of the
paper \cite{NaR} related to dihedral group fusion rules that will
be used later. In Section \ref{pruebas} we give the proofs of
Theorems \ref{soluble} and \ref{gp-ttic}.

\section{Preliminaries}\label{preliminaries}

\subsection{Fusion categories}


Let $\C$ be a fusion category. We shall denote by $\Irr(\C)$ the
set of isomorphism classes of simple objects of $\C$ and by
$G(\C)$ the group of isomorphism classes of invertible objects of
$\C$. For an object $X$ of $\C$, we shall indicate by $\C[X]$ the fusion subcategory generated by $X$ and by $G[X]$ the subgroup of $G(\C)$ consisting of invertible objects $g$ such that $g \otimes X \simeq X$.

If $\mathcal D$ is another fusion category, $\C$ and $\D$
are \emph{Morita equivalent} if $\D$ is equivalent to the dual
$\C^*_{\mathcal M}$ with respect to an indecomposable module
category $\mathcal M$. Recall that $\C$ is called \emph{pointed}
if all its simple objects are inver\-tible and it is called
\emph{group-theoretical} if it is Morita equivalent to a pointed
fusion category.

There is a canonical faithful grading $\C = \oplus_{g \in U(\C)}\C_g$, with trivial component $\C_e = \C_{\ad}$, where $\C_{\ad}$ is the \emph{adjoint
subcategory} of $\C$, that is, the fusion subcategory generated by $X \otimes X^*$, where $X$ runs through the
simple objects of $\C$. The group
$U(\C)$ is called the \emph{universal grading group} of $\C$. $\C$ is called nilpotent if the upper central series $\dots \subseteq \C^{(n+1)} \subseteq \C^{(n)} \subseteq \dots \subseteq \C^{(0)} = \C$ converges to $\vect_k$, where $\C^{(i)} : = (\C^{(i-1)})_{\ad}$, $i \geq 1$. See \cite{gel-nik}.

A \emph{weakly group-theoretical} fusion category is a fusion
category $\C$ which is Morita equivalent to a nilpotent fusion
category. If $\C$ is Morita equivalent to a cyclically nilpotent
fusion category, then $\C$ is called \emph{solvable}. We refer the
reader to \cite{ENO, ENO2} for further definitions and facts about
fusion categories.

\subsection{Braided fusion categories}
Let $\C$ be a \emph{braided} fusion category, that is, $\C$ is
equipped with  natural isomorphisms $c_{X,Y} : X \otimes Y \rightarrow Y \otimes X$, $X, Y \in \C$, satisfying the hexagon
axioms.
Recall that $\C$ is called \emph{premodular} if it is also spherical, that is, $\C$ has a pivotal structure such that left and right categorical dimensions coincide.
Equivalently, $\C$ is premodular if it is endowed with a compatible ribbon structure \cite{bruguieres, Mu1}.

We say that the objects $X$ and $Y$ of a braided fusion category $\C$ centralize each other if
$c_{Y,X} c_{X,Y} = \id_{X\otimes Y}$. The \emph{centralizer} $\D'$
of a fusion subcategory $\D \subseteq \C$ is defined to be the full
subcategory of objects of $\C$ that centralize every object of $\D$.
The centralizer $\D'$
results a fusion subcategory of $\C$.

The \emph{Müger (or symmetric) center}  $Z_2(\C)$ of $\C$ is $Z_2(\C) = \C'$; this is a symmetric fusion subcategory of $\C$ whose objects are called central, dege\-nerate or transparent.
A braided fusion category $\C$ is called \emph{non-degenerate} if its Müger center $Z_2(\C)$ is trivial.
A \emph{modular} category is a non-degenerate premodular category $\C$.

\begin{remark}\label{spherical} Recall that a fusion category $\C$ is called pseudo-unitary if $\dim \C = \FPdim \C$,
where $\dim \C$ is the global dimension of $\C$ and $\FPdim \C$ is
the Frobenius-Perron dimension of $\C$. If $\C$ pseudo-unitary
then $\C$ has a canonical spherical structure with respect to
which categorical dimensions of all simple objects coincide with
their Frobenius-Perron dimensions \cite[Proposition 8.23]{ENO}.

In particular, this holds for any weakly integral fusion category,
because it is automatically pseudo-unitary \cite[Proposition
8.24]{ENO}.
Hence every weakly integral non-degenerate fusion
category is canonically a modular category.
\end{remark}

\section{Some families of examples}\label{examples}

\subsection{Examples of fusion categories with Frobenius-Perron dimensions $\leq 2$}

In this subsection we discuss examples of weakly integral fusion categories with
Frobenius-Perron dimensions of simple objects $\leq 2$ that appear in the literature.

\begin{ejem}
Consider  a Hopf algebra
$H$ fitting into an abelian exact sequence:
\begin{equation}\label{exacta} k\rightarrow k^\Gamma \rightarrow H \rightarrow
k\mathbb Z_2 \rightarrow k,
\end{equation} where $\Gamma$ is a finite group.
Let $\C = \Rep H$.  Then $\cd (\C) \subseteq \{1,2\}$ and equality holds if the
associated action of $\mathbb Z_2$ on $\Gamma$ is not trivial.

All these examples are group-theoretical, in view of
\cite[Theorem 1.3]{gp-ttic}.
Observe that, as a consequence of \cite[Theorem
6.4]{BN}, any cosemisimple Hopf algebra $H$ such that $\cd(\C) \subseteq
\{ 1, 2\}$ is group-theoretical if $\C = \C_{\ad}$. See \cite[Theorem 6.4]{NP}.

Non-trivial examples of cosemisimple Hopf algebras fitting into an exact sequence \eqref{exacta}  are given by the Hopf algebras $$\A^*_{4m},  \B^*_{4m} \quad m\geq 2,$$
of dimension $4m$, due to Masuoka \cite{mas-cocycle}.  In these cases, $\Gamma$ is a dihedral group.
\end{ejem}

\begin{ejem}
Let $\C = \TY (G, \chi, \tau)$ be the Tambara-Yamagami category
associated to a finite (necessarily abelian) group $G$, a
symmetric non-degene\-rate bicharacter $\chi : G\times G \rightarrow
k^\times$ and an element $\tau\in k$ satisfying $|G|\tau^2 = 1$
\cite{TY}. This is a fusion category with isomorphism classes of
simple objects parameterized by the set $G\cup\{X\}$, where $X
\notin G$, obeying the fusion rules
\begin{equation}\label{ty}
g \otimes h = gh, \quad g, h\in G,\quad X \otimes X = \oplus_{g\in G} g.
\end{equation}


We have $\cd (\C) = \{1,2\}$ if and only if $G$ is of order $4$.
Therefore, in this case $\FPdim \C = 8$.

If $G\simeq \mathbb Z_4$, there are two possible fusion categories
$\C$. None of them is braided \cite[Theorem 1.2 (1)]{Siehler-braided}.

If $G\simeq \mathbb Z_2 \times \mathbb Z_2$ there are exactly four
classes of Tambara-Yamagami categories with irreducibles degrees
$1$ or $2$, by \cite[Theorem 4.1]{TY}. Three of them are
(equivalent to) the categories of representations of
eight-dimensional Hopf algebras: the dihedral group algebra of
order $8$, the quaternion group algebra, and the Kac-Paljutkin
Hopf algebra $H_8$. The remaining fusion ca\-tegory, which has the
same $\chi$ as $H_8$ but $\tau = -1/2$, is not realized as the
fusion category of representations of a Hopf algebra. Since in
this case $G$ is an elementary abelian $2$-group all of this
categories admit a braiding, by \cite[Theorem 1.2 (1)]{Siehler-braided}.

All the fusion categories in this example are group-theoretical.
In fact, by \cite[Lemma 4.5]{GNN}, for any symmetric
non-degenerate bicharacter $\chi:G\times G \rightarrow
k^{\times}$, $G$ contains a Lagrangian subgroup with respect to
$\chi$. Therefore $\TY (G, \chi, \tau)$ is group-theoretical, by
\cite[Theorem 4.6]{GNN}.

\end{ejem}

\begin{ejem}
Recall that a near-group category is a fusion category with exactly one isomorphism class of non-invertible simple object.
In the notation of \cite{Siehler-braided}, the fusion rules of $\C$ are determined by a pair
$(G, \kappa)$, where $G$ is the group of invertible objects of $\C$ and $\kappa$ is a nonnegative
integer. Letting $\Irr(\C) = G \cup \{X\}$, where $X$ is non-invertible, we have the relation
\begin{equation}
 X\otimes X = \oplus_{g \in G} g \oplus \kappa X.
\end{equation}

Near-group categories with fusion rule $(G, 0)$ for some finite
group $G$ are thus Tambara-Yamagami categories, discussed in the previous example.
Let us consider near-group categories with fusion rule $(G, \kappa)$ for some finite
group $G$ and a positive integer $\kappa$.

We have $\cd (\C) = \{1,2\}$ if and only if $G$ is of order $2$ and $\kappa = 1$, that means $\C$ is of type $(\mathbb Z_2, 1)$.
Therefore, in this case $\FPdim \C = 6$ and since $\kappa > 0$, then $\C$ is group-theoretical, by \cite[Theorem 1.1]{EGO}. By \cite[Theorem 1.5]{Thornton}, there are up to equivalence exactly two non-symmetric braided near-group categories with fusion rule $(\mathbb Z_2, 1)$.
\end{ejem}

\begin{ejem} Examples of a weakly integral braided fusion categories which are not integral and Frobenius-Perron dimensions of simple objects are $\leq 2$ are given by the Ising categories, studied in \cite[Appendix B]{DGNOI}.
In this case, there is a unique non-invertible simple object $X$ with $X^{\otimes 2} = \1 \oplus a$,
where $a$ generates the group of invertible objects, isomorphic to $\mathbb Z_2$ (note that these are also Tambara-Yamagami ca\-tegories). We have here $\cd(\C) = \{ 1, \sqrt 2 \}$ and $\FPdim \C = 4$.
Every braided Ising
category is modular \cite[Corollary B.12]{DGNOI}. 

Other examples come from braided fusion categories with generalized Tambara-Yamagami fusion rules of type $(G, \mathbb Z_2)$, where $G$ is a finite group. See \cite{liptrap}. In these examples, $\C$ is not pointed, the group of invertible objects is $G$, and $\mathbb Z_2 \simeq \Gamma \subseteq G$ is a subgroup such that $X \otimes X^* \simeq \oplus_{h \in \Gamma} h$, for all non-invertible object $X$ of $\C$. Hence we also have $\cd(\C) = \{ 1, \sqrt 2 \}$.

Since they are not integral, these examples are not group-theoretical.
\end{ejem}

\begin{ejem} Let   $\C$ be a braided group-theoretical fusion
category. Then $\C$ is an equivariantization of a pointed fusion
category, that is, $\C \simeq \D^G$, where $\D$ is a pointed
fusion category and $G$ is a finite group acting on $\D$ by tensor
autoequivalences \cite{NNW}. In this case, $\C$ contains the
category $\Rep G$ of finite-dimensional representations of $G$ as
a fusion subcategory.

Suppose that $\cd(\C) = \{1, p\}$, where $p$ is any prime number.
Then also $\cd(G) \subseteq \{1, p\}$.  In particular, the group
$G$ must have a normal abelian $p$-complement; moreover, either
$G$ contains an abelian normal subgroup of index $p$ or the center
$Z(G)$ has index $p^3$. See \cite[Theorems 6.9, 12.11]{isaacs}.
\end{ejem}

\subsection{Fusion rules of dihedral type}\label{fusion_D_n}

Let $D_n$ be the dihedral group of order $2n$, $n\geq 1$. Recall
that $D_n$ has a presentation by generators $t,z$ and relations
$t^2 = 1 = z^n$, $tz = z^{-1}t$.


The following proposition describes the fusion rules of $\Rep D_n$
(\textit{c.f.} \cite{mas-cocycle}).

\begin{proposition}\label{D_n}
\begin{enumerate}
\item Suppose $n$ is odd. Then the isomorphisms classes of simple
objects of $\Rep D_n$ are represented by $2$ invertible objects,
$\textbf{1}$ and $g$, and $r = (n-1)/2$ simple objects $X_1,
\ldots, X_r$, of dimension $2$, such that
    \begin{align*}
    & g\otimes X_i = X_i = X_i\otimes g,  \qquad  \forall i=1, \ldots, r, \\
    & X_i\otimes X_j = \left\{
    \begin{array}{ll}  X_{i+j}\oplus X_{|i-j|}, \quad  & \text{if} \quad i+j \leq r, \\
                       X_{n-(i+j)}\oplus X_{|i-j|}, \quad & \text{if} \quad i+j > r;
    \end{array}
    \right.
    \end{align*}
where $X_0 = \textbf{1}\oplus g$.

\item Suppose $n$ is even, that is $n = 2m$. Then the isomorphisms
classes of simple objects of $\Rep D_n$ are represented by $4$
invertible objects, $\textbf{1}$, $g$, $h$, $f = gh$, and
$m-1$ simple objects $X_1, \ldots, X_{m-1}$, of dimension $2$,
such that
 \begin{align*}
    & g\otimes X_i =  X_i = X_i\otimes g, \qquad  \forall i=1, \ldots, m-1, \\
    & h\otimes X_i =  X_{m-i} = X_i\otimes h, \qquad  \forall i=1, \ldots, m-1, \\
    & X_i\otimes X_j = \left\{ \begin{array}{ll} X_{i+j}\oplus X_{|i-j|}, \quad  & \text{if} \quad i+j \leq m, \\
                                            X_{2m-(i+j)}\oplus X_{|i+j|}, \quad & \text{if} \quad i+j >
                                            m;
\end{array}
\right.
 \end{align*}
where $X_0 = \textbf{1}\oplus g$ and $X_m = h\oplus f$.
\end{enumerate}
In particular, the group of invertible objects in $\Rep D_n$ is isomorphic to $\mathbb Z_2$ if $n$ is odd, and to $\mathbb Z_2\times\mathbb Z_2$ if $n$ is even.
\end{proposition}

\begin{remark}\label{ndivide4}
Suppose that $4$ divides $n = 2m$. Then $X_{m/2}$ is
fixed under (left and right) multiplication by all invertible objects of $\Rep D_n$.
\end{remark}

Let $\C$ be a fusion category with $\cd (\C) = \{1, 2\}$. Suppose that the Grothendieck ring of $\C$ is commutative
(for example, this is the case if $\C$ is braided). Assume in addition that the following conditions hold:

\begin{enumerate}
\item[(a)] All objects are self-dual, that is $X \simeq X^*$, for every object $X$ of $\C$.

\item[(b)] $\C$ has a faithful simple object.
\end{enumerate}
Then, it is shown in \cite[Theorem 4.2]{NaR} that $\C$ is
Grothendieck equivalent to $\Rep D_n$.  Moreover, $\C$ is
necessarily group-theoretical.

\medbreak It is possible to remove the assumption that all the
objects are self-dual but it is still necessary the condition of
self-duality on the faithful simple object. Namely, suppose that
$\C$ is not self-dual, but satisfies
\begin{enumerate}
\item[(b')] $\C$ has a faithful self-dual simple object.
\end{enumerate}

In this case $\C$ is still group-theoretical and it is
Grothendieck equivalent to $\Rep \widetilde D_n$, $n$ odd. See
\cite[Remark 4.4]{NaR}. Here $\widetilde D_n$ is the generalized quaternion
(binary dihedral) group of order $4n$, that is, the group
presented by generators $a, s$, with relations $a^{2n} = 1$, $s^2
= a^n$, $s^{-1}as = a^{-1}$. (Observe that for $n$ odd,
$\widetilde D_n$ is isomorphic to the semidirect product $\mathbb
Z_n \rtimes \mathbb Z_4$, with respect to the action given by
inversion, considered in \cite{NaR}. For even $n$, $\Rep
\widetilde D_n$ is Grothendieck equivalent to $\Rep D_{2n}$, while
$\mathbb Z_n \rtimes \mathbb Z_4$ has no faithful representation
of degree $2$.)


\begin{lemma}\label{centers} Let $n \geq 2$. Then $(\Rep \widetilde D_{n})_{\ad}
= \Rep D_{n}$. In addition,
\begin{equation*}(\Rep D_{n})_{\ad} = \left\{ \begin{array}{ll} \Rep D_{n/2}, \quad  & \text{if} \quad n \quad \text{is even}, \\
                                            \Rep D_{n}, \quad  & \text{if}\quad n \quad \text{is odd}.
\end{array}
\right.
\end{equation*}
\end{lemma}

\begin{proof} Recall that when $\C = \Rep G$, where $G$ is a finite group,
then $\C_{\ad} = \Rep G/Z(G)$ \cite{gel-nik}. The first claim
follows from the fact that the center of $\tilde D_n$ equals $\{1,
s^2\} \simeq \mathbb Z_2$. On the other hand, the center $Z(D_n)$
is trivial if $n$ is odd, and equals $\{ 1, z^{n/2}\} \simeq \mathbb
Z_2$ if $n$ is even.
 This implies the second claim and finishes the proof of the lemma.
\end{proof}

\section{Proof of the main results}\label{pruebas}

In this section we shall prove Theorems \ref{soluble} and
\ref{gp-ttic}.

\begin{proposition}\label{equiv}
Let $\C$ be a premodular fusion category. Suppose $\C$ has an
invertible object $g$ of order $n$ and a simple object $X$ such
that
\begin{flalign} \label{(1)}& g\otimes X = X, \textrm{ and } & \\
\label{(2)}& g \textrm{ centrali\-zes } X.&
\end{flalign}
Then we have
\begin{enumerate} \item[(i)] $\C$ is an equivariantization by the cyclic group
$\mathbb Z_n$ of a fusion category $\widetilde \C$.
\item[(ii)] If $g \in \C'$, then $\widetilde \C$ is braided.
\end{enumerate}
\end{proposition}

\begin{proof}
Condition \eqref{(1)} ensures the existence of a fiber functor on the fusion category $\C [g]$ generated by
$g$. Then $\C [g]$ is equivalent to
$\Rep \mathbb Z_n$ as fusion categories.

Moreover, they are equivalent as braided fusion categories.
Indeed, \eqref{(1)} implies $\C[g]\subseteq \C[X]$ and therefore
$\C[g]\subseteq Z_2(\C[X])$, by \eqref{(2)}. Hence $\C[g]$ is
symmetric. Then the only possible twists in $\C$ are $\theta_h =
1$ and $\theta_h = -1$ for all $h\in\langle g \rangle$. But
$\theta_h$ is not equal to $-1$ since $h$ centralizes $X$ and
$h\otimes X = X$ \cite[Lemma 5.4]{Mu}. Then $\theta_h = 1$ for all
$h\in\langle g \rangle$. Therefore $\C [g] \simeq \Rep \mathbb
Z_n$ as braided fusion categories, as claimed.

Let $\Gamma = \langle g \rangle \subseteq G(\C)$.  It follows from \cite[Theorem 4.18 (i)]{DGNOI} that the de-equivariantization $\widetilde \C = \C_{\Gamma}$ of $\C$ by
$\Gamma$ is a fusion category and there is a canonical equivalence $\C\simeq {\widetilde \C}^{\Gamma}$ between the category $\C$ and the
$\Gamma$-equivariantization of $\widetilde \C$, which shows (i).

Furthermore, if $g \in  \C'$ then $\widetilde \C$ is braided and the  equivalence
$\C \simeq {\widetilde \C}^{\Gamma}$ is of braided fusion categories \cite{bruguieres, Mu} (see also \cite[Theorem 4.18 (ii)]{DGNOI}).
Thus we get (ii).
This proves the proposition.
\end{proof}

\begin{lemma}\label{generadores}
Let $\C$ be a fusion category with commutative Grothendieck ring. Suppose  that $\C = \C_{\ad}$. If $\D_1,
\ldots, \D_s$ are fusion subcategories that generate $\C$ as a
fusion category, then $\D_1^{(m)}, \ldots, \D_s^{(m)}$ generate
$\C$ as a fusion category, $\forall m\geq 0$.
\end{lemma}

\begin{proof} Since $\D_1, \ldots, \D_s$  generate $\C$, then $(\D_1)_{\ad}, \ldots,
(\D_s)_{\ad}$ generate $\C$. In fact, let $X$ be a simple object
of $\C$. There exist simple objects $X_{i_1}, \ldots, X_{i_t}$,
with  $X_{i_l} \in \D_{i_l}$, $1 \leq i_1, \dots, i_t \leq s$,
such that $X$ is a direct summand  of $X_{i_1}\otimes \ldots
\otimes X_{i_t}$. Then $X\otimes X^*$ is a direct summand  of
$$X_{i_1}\otimes \ldots \otimes X_{i_t}\otimes X_{i_t}^*\otimes
\ldots \otimes X_{i_1}^* \simeq (X_{i_1}\otimes X_{i_1}^*)\otimes
\ldots \otimes (X_{i_t}\otimes X_{i_t}^*),$$ where we have used
that $\C$ has a commutative Grothendieck ring.
Notice that the object in the right hand side belongs to the fusion subcategory generated by $(\D_1)_{\ad}, \ldots,
(\D_s)_{\ad}$.
Since $X$ was arbitrary, it follows that  $(\D_1)_{\ad}, \ldots,
(\D_s)_{\ad}$ gene\-rate $\C_{\ad}$. But $\C = \C_{\ad}$ by
assumption, then we have proved that $(\D_1)_{\ad}, \ldots,
(\D_s)_{\ad}$ generate $\C$.  The statement follows from this by
induction on $n$, since $\D_j^{(n)} = (\D_j^{(n-1)})_{\ad}$, for
all $j = 1, \ldots s$, $n\geq 1$.
\end{proof}

\subsection{Braided fusion categories with irreducible degrees $1$ and $2$ }

Throughout this subsection $\C$ is a braided fusion category with
$\cd (\C) = \{1,2\}$. We regard $\C$ as a premodular category with
respect to its canonical spherical structure. See Remark
\ref{spherical}.

\begin{remark}\label{orderG[X]}
Note that $G[X]\neq \textbf{1}$, for all
$X$ such that $\FPdim X = 2$. Moreover, $|G[X]| = 2$ or $4$. In
particular the (abelian) group $G(\C)$ is not trivial.
\end{remark}

\begin{proposition}\label{equi_g} Let $g$ be a non-trivial invertible object such that $g^2 = 1$ and $\theta_g = 1$.
Assume that $g$ generates the Müger center $\C'$ of $\C$ as a
fusion category.
Then $\C$ is the equivariantization of a modular fusion category
$\widetilde \C$ by the group $\mathbb Z_2$. Furthermore
$\cd(\widetilde \C) \subseteq \{1,2\}$.
\end{proposition}

\begin{proof}
By assumption $\C'\simeq \Rep \mathbb Z_2$ is tannakian. Then the
de-equivarianti\-zation $\widetilde \C$ of $\C$ by $\C'$ is a
modular category and there is an action of $\mathbb Z_2$ on
$\widetilde \C$ such that $\C \simeq \widetilde \C ^{\mathbb Z_2}$
\cite{bruguieres, Mu}.
Since $\cd(\widetilde \C ^{\mathbb Z_2}) = \cd(\C) = \{1,2\}$,
then $ \cd(\widetilde \C) \subseteq \{1,2\}$, by \cite[Proof of Proposition
6.2]{ENO2}, \cite[Lemma 7.2]{NP}.
\end{proof}

\begin{lemma}\label{noad}
Suppose that $\C\neq \C_{\ad}$ and $\C_{\ad}$ is solvable. Then $\C$ is
solvable.
\end{lemma}

\begin{proof}


Since $\C$ is braided, its universal grading group $U(\C)$ is
abelian \cite[Theorem 6.2]{gel-nik}. The category $\C$ is a
$U(\C)$-extension of $\C_{\ad}$ and an extension of a solvable
category by a solvable group is again solvable \cite[Proposition
4.5 (i)]{ENO2}. Then $\C$ is solvable, as claimed.
\end{proof}

\begin{lemma}\label{ad}
Assume $\C = \C_{\ad}$. Then $\FPdim \C' \geq 2$.
\end{lemma}

\begin{proof}
Suppose on the contrary that $\FPdim \C' = 1$, that is, $\C$ is
modular. Then, by \cite[Theorem 6.2]{gel-nik}, $U(\C)\simeq
\widehat{G(\C)} \simeq G(\C)$. By Remark \ref{orderG[X]}, $\C_{\ad}\subsetneq \C$, against the
assumption. Hence $\FPdim \C'\geq 2$, as claimed.
\end{proof}

\begin{lemma}\label{lema-dn} Suppose $\C$ is generated by a simple object $X$ such
that $X\simeq X^*$ and $\FPdim X = 2$. Then we have
\begin{enumerate}
\item[(i)] $\C$ is not modular.
\end{enumerate}
Assume $\C = \C_{\ad}$. Then we have in addition
\begin{enumerate}
\item[(ii)] There is a group isomorphism $G(\C)\simeq \mathbb
Z_2$.
\item[(iii)] $G(\C)\subseteq \C'$.
\end{enumerate}

\end{lemma}

\begin{proof}
By \cite[Theorem 4.2; Remark 4.4]{NaR}, $\C$ is Grothendieck
equivalent to $\Rep D_n$ or $\Rep \widetilde D_{2n+1}$, for some $n\geq
1$.
Since the universal grading group is a Grothendieck invariant, then in
the first case $U(\C)$ is isomorphic to $\mathbb Z_2$ if $n$ is
even and is trivial if $n$ is odd. But $G(\C)$, which is also a
Grothendieck invariant, is isomorphic to $\mathbb Z_2 \times
\mathbb Z_2$ if $n$ is even and is isomorphic to $\mathbb Z_2$ if
$n$ is odd, by Proposition \ref{fusion_D_n}. Then $U(\C)$ is not
isomorphic to $\widehat{G(\C)}$, for any $n$. Therefore $\C$ is
not modular, by \cite[Theorem 6.2]{gel-nik}. Similarly, if $\C$ is
Grothendieck equivalent to $\Rep \widetilde D_{2n+1}$, we have $U(\C) \simeq \mathbb Z_2$ and $G(\C) \simeq
\mathbb Z_4$. Hence $\C$ is not modular in this
case neither. This shows (i).

Notice that the assumption $\C = \C_{\ad}$ implies that $\C$ is
Grothendieck equivalent to $\Rep D_n$, for some $n$ odd. Then
 (ii)  follows immediately from the fusion rules of $\Rep D_n$, with $n$ odd (see Proposition
 \ref{fusion_D_n}). Since, by (i), $\C'$ is not trivial, then $G(\C') \neq
 \1$, because $\cd(\C') \subseteq \{ 1, 2\}$ (\textit{c.f.} Remark \ref{orderG[X]}). By part (i), $G(\C') =
 G(\C)$ and (iii) follows.
\end{proof}

\begin{remark} \label{d_n_impar}If $\C$ is a fusion category as in Lemma \ref{lema-dn},  then the assumption $\C =
\C_{\ad}$ is equivalent to saying that $\C$ is Grothendieck equivalent
to $\Rep D_n$, for some $n \geq 1$ \emph{odd}.
\end{remark}

\begin{lemma}\label{genconad}  Suppose that $\C = \C_{\ad}$.
Then $\C$ is generated by fusion subca\-tegories $\D_1, \dots,
\D_s$, $s \geq 1$, where $\D_i$ is Grothendieck equivalent to
$\Rep D_{n_i}$ and $n_i$ is an odd natural number, for all $i = 1,
\dots, s$.
\end{lemma}

\begin{proof} Let $\C = \C[X_1, \ldots, X_s]$ for some  simple objects $X_1, \ldots,
X_s$. Let $\D_i = \C[X_i]$ be the fusion subcategory generated by $X_i$, $i = 1, \ldots, s$.
By Lemma \ref{generadores}, $(\D_1)_{\ad}, \ldots, (\D_s)_{\ad}$ generate $\C$ as a fusion category.
Hence, it is enough to consider only those simple objects $X_i$ whose Frobenius-Perron dimension equals $2$ (otherwise, $\FPdim X_i = 1$ and $X_i\otimes X_i^* \simeq \1$).

Moreover, iterating the application of Lemma \ref{generadores}, we may further assume that $|G[X_i]| = 2$, for all $i = 1,  \dots, s$. Thus we have a decomposition
$X_i \otimes X_i^* \simeq \1 \oplus g_i \oplus X_i'$, where $G[X_i] = \{ \1, g_i\}$ and $X_i'$ is a self-dual simple object of Frobenius-Perron dimension $2$.
Since $X_i\otimes X_i^*$ generates $(\D_i)_{\ad}$, the above reductions allow us to assume that $\D_i = \C[X_i]$ with $X_i$ simple objects
of $\C$ such that $\FPdim X_i = 2$ and $X_i \simeq X_i^*$,
$\forall i = 1, \ldots, s$.


We claim that we can choose the $X_i$'s in such a way
that $(\D_i)_{\ad}\simeq \D_i$. By \cite[Theorem 4.2; Remark
4.4]{NaR},  $\D_i$ is Grothendieck equivalent to $\Rep D_{n_i}$ or
to $\Rep \widetilde D_{2n_i+1}$.
Iterating the application of Lemma \ref{generadores} and using Lemma \ref{centers}, we obtain that
$\C = \C[\D_1, \ldots, \D_s]$, with $\D_j$ a fusion subcategory of
$\C$ Grothendieck equivalent to $\Rep D_{n_j}$, $n_j$ odd,
for all $j = 1, \ldots, s$, as we wanted.
\end{proof}

%
%
%
%

\subsection{Proof of Theorems \ref{soluble} and \ref{gp-ttic}} Let $\C$ be a weakly integral fusion category.
It follows from \cite[Theorem 3.10]{gel-nik} that either $\C$ is
integral, or  $\C$ is a $\mathbb Z_2$-extension of a fusion
subcategory $\D$. In particular, if $\C = \C_{\ad}$, then $\C$ is
necessarily integral.

\begin{lemma}\label{prod-simples-categorico}
Let $\C$ be fusion category and let $X, X'$ be simple objects of
$\C$. Then the following are equivalent:
\begin{enumerate} \item[(i)] The tensor product
$X^*\otimes X'$ is simple.

\item[(ii)] For every simple object $Y \neq \1$ of $\C$, either $m(Y, X\otimes X^*)
= 0$ or $m(Y, X'\otimes X'^*) = 0$. \end{enumerate}


In particular, if
$X^*\otimes X'$ is not simple, then $\C[X]_{\ad} \cap
\C[X']_{\ad}$ is not trivial.
\end{lemma}

\begin{proof} The equivalence between (i) and (ii) is proved in \cite[Lemma 6.1]{BN} in the case where $\C$ is the category of (co)representations of a semisimple Hopf algebra.
Note that the proof \textit{loc. cit.} works in this more general context as well.
\end{proof}

\begin{proof}[Proof of Theorem \ref{soluble}]
The proof is by induction on $\FPdim \C$. As pointed out at the
beginning of this subsection, if $\C$ is not integral, then it is
a $\mathbb Z_2$-extension of a fusion subcategory $\D$. Since $\D$
also satisfies the assumptions of the theorem, then $\D$ is
solvable, by induction. Hence $\C$ is solvable as well.

We may thus assume that $\C$ is integral. Therefore $\cd(\C) =
\{1, 2 \}$ and the results of the previous subsection apply. By
Lemma \ref{noad}, we may assume that $\C = \C_{\ad}$. Then it
follows from Lemma \ref{genconad} that $\C = \C[\D_1, \ldots,
\D_s]$, with $\D_j$ Grothendieck equivalent to $\Rep D_{n_j}$,
$n_j$ odd, $\forall j = 1, \ldots, s$.

By Lemma \ref{lema-dn}, $G(\D_j) = \{\textbf{1}, g_j\}$, $\forall
j = 1, \ldots, s$. We claim that $g_i = g_j$ $\forall 1\leq i, j
\leq s$. Indeed, let $\D_j = \C[X^{(j)}]$, where $X^{(j)} =
X_1^{(j)}$ in the notation of Proposition \ref{D_n}. Then we have
$(X^{(j)})^{\otimes 2} = \textbf{1}\oplus g_j \oplus X_2^{(j)}$.
Fix $1\leq i, j \leq s$. Since $\C$ has no simple objects of
Frobenius-Perron dimension $4$ then $g_i = g_j$ or
$X_2^{(j)}\simeq X_2^{(i)}$, by Lemma
\ref{prod-simples-categorico}. In the first case we are done. In
the second case, we note that $\{1, g_j\} = G[X_2^{(j)}] =
G[X_2^{(i)}] = \{1, g_i\}$. Then $g_j = g_i$, as claimed. Let $g = g_j = g_i$.

By Lemma \ref{lema-dn},  $g\in \D_i'$, for all $i = 1, \ldots, s$.
Since $\D_i$, $1\leq i \leq s$, generate $\C$ then $g \in \C'$. It
follows from Theorem \ref{equiv} (ii) that $\C$ is the
equivariantization by $\mathbb Z_2$ of a braided fusion category
$\widetilde \C$. In particular, $\FPdim \widetilde \C = \FPdim \C
/ 2$ and $\cd (\widetilde \C) \subseteq \{1,2\}$, by \cite[Proof of
Proposition 6.2 (1)]{ENO2}, \cite[Lemma 7.2]{NP}. By inductive
hypothesis, $\widetilde \C$ is solvable. Then $\C$, being the
equivariantization of a solvable fusion category by a solvable
group is itself solvable \cite[Proposition 4.5 (i)]{ENO2}.
\end{proof}

\begin{theorem}\label{morita-ccad} Let $\C$ be a weakly integral braided fusion category that $\FPdim X \leq 2$ for all simple object $X$ of $\C$.
Assume in addition that $\C = \C_{\ad}$. Then $\C$ is tensor
Morita equivalent to a pointed fusion category $\C(A \rtimes \mathbb Z_2,
\tilde \omega)$, where $A$ is an abelian group endowed with an action of $\mathbb Z_2$ by group automorphisms, and $\tilde
\omega$ is a certain $3$-cocycle on the semidirect product $A \rtimes \mathbb Z_2$. \end{theorem}

\begin{proof} The assumption $\C = \C_{\ad}$ implies that $\C$ is
integral. Hence we may assume that $\cd(\C) = \{ 1, 2\}$. By Lemma
\ref{genconad}, $\C$ is generated by fusion subcategories $\D_1,
\dots, \D_s$, $s \geq 1$, where $\D_i$ is Grothendieck equivalent
to $\Rep D_{n_i}$ and $n_i$ is an odd natural number, for all $i =
1, \dots, s$. Furthermore, as in the proof of Theorem
\ref{soluble}, the assumption that $\C = \C_{\ad}$ implies that
$G(\D_i) = \{ \1, g\}$, for all $1\leq i \leq s$, and $\C[g]
\simeq \Rep \mathbb Z_2$ is a tannakian subcategory of the M\"
uger center $\C'$. So that $\C \simeq \tilde \C^{\mathbb Z_2}$ is
an equivariantization of a braided fusion category $\tilde \C$.

Equivariantization under a group action gives rise to exact
sequences of fusion categories \cite[Subsection
5.3]{tensor-exact}. In our situation we have an exact sequence
of braided tensor functors
\begin{equation}\label{sec-c}\Rep \mathbb Z_2 \to \C \overset{F}\to \tilde \C.\end{equation} In addition,
since $\C[g] \subseteq \D_i$,  then \eqref{sec-c} induces by
restriction an exact sequence
\begin{equation}\label{sec-di}\Rep \mathbb Z_2 \to \D_i \to \tilde \C_i,\end{equation}
for all $i = 1, \dots, s$, where $\tilde \C_i$ is the essential
image of $\D_i$ in $\tilde \C$ under the functor $F$. Hence
$\tilde \C_i$ is a fusion subcategory of $\tilde \C$, for all $i$,
and moreover $\tilde \C_1, \dots, \tilde \C_s$ generate $\tilde
\C$ as a fusion category.
Note in addition that $\cd(\tilde \C), \cd(\tilde \C_i) \subseteq
\{ 1, 2\}$, for all $i = 1, \dots, s$.
On the other hand, exactness of the sequence
\eqref{sec-di} implies that $2n_i = \FPdim \D_i = 2 \FPdim \tilde
\C_i$ \cite[Proposition 4.10]{tensor-exact}. Hence $\FPdim \tilde
\C_i = n_i$ is an odd natural number.

Since $\tilde \C_i$ is an integral braided fusion category, then
the Frobenius-Perron dimension of every simple object of $\tilde
\C_i$ divides the Frobenius-Perron dimension of $\tilde \C_i$
\cite[Theorem 2.11]{ENO2}. Thus we get that $\FPdim Y = 1$, for
all $Y \in \Irr (\tilde \C_i)$. That is, $\tilde \C_i$ is a
pointed braided fusion category, for all $i = 1, \dots, s$. Since
$\tilde \C_1, \dots, \tilde \C_s$ generate $\tilde \C$ as a fusion
category, then $\tilde \C$ is also pointed. Therefore $\tilde \C
\simeq \C(A, \omega)$ as fusion categories, where $A$ is an
abelian group and $\omega \in H^3(A, k^{\times})$.

Group actions on pointed categories were classified by
Tambara \cite{tambara}. In view of \cite[Theorem 4.1]{tambara} and \cite[Proposition 3.2]{nik},
the fusion category $\C \simeq \tilde \C^{\mathbb Z_2}$ is tensor
Morita equivalent to a pointed category $\C(A \rtimes \mathbb Z_2,
\tilde \omega)$, where the semidirect product $A \rtimes \mathbb
Z_2$ is with respect of the induced action of $\mathbb Z_2$ on the
group $A$ of invertible objects of $\tilde \C$, and $\tilde
\omega$ is a certain $3$-cocycle on $A \rtimes \mathbb Z_2$.
\end{proof}

\begin{proof}[Proof of Theorem \ref{gp-ttic}.]
The proof is an immediate consequence of Theorem
\ref{morita-ccad}.
\end{proof}

\begin{remark} Let $\C$ be a braided fusion category such that  $\cd(\C) = \{ 1, 2
\}$. Suppose that $\C$ is nilpotent. By \cite[Theorem 1.1]{DGNO}
$\C$ admits a unique decomposition (up to the order of factors)
into a tensor product $\C_1 \boxtimes \dots \boxtimes \C_m$, where
$\C_i$ are braided fusion categories of Frobenius-Perron dimension
$p_i^{m_i}$, for some pairwise distinct prime numbers $p_1, \dots,
p_m$.
Then $\C_i$ is an integral braided fusion category, for all $i =
1, \dots, m$, and by \cite[Theorem 2.11]{ENO2}, we get that $\C_i$
is pointed whenever $p_i > 2$. Hence $\C \simeq \C_1 \boxtimes
\mathcal B$ as braided fusion categories, where $\C_1$ is a
braided fusion category of Frobenius-Perron dimension $2^m$ such
that $\cd(\C_1) = \{ 1, 2 \}$, and $\mathcal B$ is a pointed
braided fusion category. \end{remark}

\bibliographystyle{amsalpha}

\end{document}